\documentclass[11pt]{amsart} 
\usepackage{amssymb,amsthm,amsmath,latexsym,enumerate,graphicx,bbm,mathptmx,microtype,tikz,multicol,ulem,soul,thmtools}
\usepackage{tikz,enumerate,amsmath,amssymb}
\usetikzlibrary{arrows}

\newtheorem{theorem}{Theorem}[section]

\newtheorem{proposition}[theorem]{Proposition}
\newtheorem{lemma}[theorem]{Lemma}

\newtheorem*{definition}{Definition}

\theoremstyle{definition}
\newtheorem{example}{Example}

\newcommand{\R}{\mathbb{R}}
\newcommand{\A}{\mathcal{A}}
\newcommand{\Z}{\mathbb{Z}}

\newcommand{\B}{\mathcal{B}}

\DeclareMathOperator*{\minimum}{min}
\DeclareMathOperator*{\fix}{fix}

\title{Unlabeled signed graph Coloring}

\author{Brian Davis}
\address{Department of Mathematics\\
         University of Kentucky\\
         Lexington, KY 40507\\
         U.S.A.}
\email{Brian.Davis@uky.edu}

\date{\today}

\begin{document}

\begin{abstract}
We extend the work of Hanlon on the chromatic polynomial of an unlabeled graph to define the unlabeled chromatic polynomial of an unlabeled signed graph. Explicit formulas are presented for labeled and unlabeled signed chromatic polynomials as summations over distinguished order-ideals of the signed partition lattice. We also define the quotient of a signed graph by a signed permutation, and show that its signed graphic arrangement is closely related to an induced arrangement on a distinguished subspace. Lastly, a formula for the number of unlabeled acyclic orientations of a signed graph is presented which recalls classical reciprocity theorems of Stanley and Zaslavsky. 
\end{abstract}
\maketitle

\section{Introduction}
Philip Hanlon considered, in the paper \cite{HanlonUnlabeled}, how to count proper colorings of an unlabeled graph, and showed that the chromatic function of an unlabeled graph was a polynomial. He additionally showed an analogous result to Stanley's reciprocity relating the number of acyclic orientations of a graph to an evaluation of its chromatic polynomial \cite{StanleyAcyclic}. We extend this work to the context of signed graphs, and give a geometric interpretation compatible with the point of view presented by Zaslavsky in \cite{zRoot}. We also introduce the quotient graph of a signed graph by an automorphism. Our principal tool is an order ideal of the signed partition lattice which partitions proper colorings by containment in flats of the type-B arrangement. 

In Proposition \ref{explicitChrom} we give an explicit expression for the chromatic function of a signed graph. Theorem \ref{explicitUnlab} states the unlabeled version, i.e., an expression in $k$ for the number of proper $k$-colorings, up to automorphism, of a signed graph. In particular we show that the unlabeled chromatic function is a polynomial. Section 2.3 describes the quotient graph $\Sigma/\beta$ of a signed graph $\Sigma$ by an automorphism $\beta$, and Theorem \ref{mainunlabeledtheorem} expresses the unlabeled chromatic polynomial as a sum over signed chromatic functions of quotient signed graphs. Theorem \ref{unlabeledacycliccounting} enumerates the unlabeled acyclic orientations of a signed graph.
\section{Background}
\subsection{Signed Graphs \& Hyperplane Arrangements}
A signed graph $\Sigma$ is a graph $G=(V,E)$ with a sign function $\sigma:E \rightarrow\{+1,-1\}$ on its edge set. Edges with distinct endpoints are called links, with one endpoint are called half-edges, and with both endpoints the same are called loops. We do not allow free loops (no endpoints). By convention the sign function $\sigma$ does not take values on half-edges. For notational convenience we identify the vertex set $V$ with $[n]:=\{1,2,\cdots,n\}$, and use the notation $V(e)$ for the endpoint set of edge $e$.

A function $\gamma\,: V\,\rightarrow\,\Z$ is a {\it{proper coloring}} of $\Sigma=(V,E,\sigma)$ when for every edge $e\in E$ with (not necessarily distinct) endpoints $i$ and $j$ the condition $\gamma(i)\neq\sigma(e)\cdot\gamma(j)$ holds, and $\gamma(i)\neq0$ whenever there is a half-edge based at $i$.

Observe that negative loops and half-edges contribute the same constraints on proper colorings,  and that no proper colorings are possible in the  presence of positive loops.
 
We are interested in the chromatic function $\chi_\Sigma$, whose input is a natural number $k$ and whose output is the number of proper colorings of $\Sigma$ taking values in $[-k,k]\cap\Z$. Though this definition is slightly different from that used in \cite{ZSignedGraphColoring}, translating between them is straightforward. Note that there are signed graphs that admit no proper coloring, i.e., $\chi_\Sigma\equiv 0$ (precisely when $\Sigma$ contains a positive loop). A nice geometric viewpoint on (signed) graph coloring in terms of lattice points and hyperplane arrangements is described in \cite{IOP} and \cite{zRoot}.

The type-$BC$ Coxeter arrangement (in dimension $n$), denoted $BC_n$, consists of the the following  hyperplanes:
\begin{align*}
h_{i,j}^+&:=\{x\in\R^n:\;x_i=x_j\}\\
h_{i,j}^-&:=\{x\in\R^n:\;x_i=-x_j\}\\
h_{i}&:=\{x\in\R^n:\;x_i=0\}
\end{align*} where $1\leq i < j \leq n$.
For a signed graph $\Sigma$ with no positive loops, the signed graphic arrangement $\B_\Sigma$ is the sub-arrangement of $BC_n$ that encodes the properness conditions of $\Sigma$-colorings. It is the collection
\[
\B_\Sigma:=\left\{h_e^{\sigma(e)}\right\}_{e\in E}
\] where we take $h_i^\pm$ to mean $h_i$. 

A map $f\;:\; [n]\rightarrow \Z$ may be associated with the point $(f(1),\dots,f(n))\in \Z^n$. A coloring $\gamma$ is proper if and only if, as a point in $\Z^n$, it avoids each hyperplane of ~$\B_\Sigma$.

Recall that associated to a (central) hyperplane arrangement $\A$ in $\R^n$ is a partially ordered set $L(\A)$ called the intersection lattice, whose elements are the geometric intersections of sub-collections of the arrangement, ordered by reverse containment. These elements are called flats of the arrangement, and they partition the space $\R^n$ by associating to each point ${\bf{x}}=(x_1,\dots,x_n)\in\R^n$ the maximal flat (minimal under inclusion) containing ${\bf{x}}$. In the context of  signed graphic arrangements, the maximal flat is equivalent to the maximal collection of equalities of type $h_e^{\sigma(e)}$ satisfied by ${\bf{x}}$.

For a signed graph $\Sigma$, we define $P(\Sigma)$, an associated sub-poset of $L(BC_n)$, as follows:
\[P(\Sigma):=\{p\in L(BC_n):\; p\nsubseteq h_e^{\sigma(e)},\; e\in E\}.\] Observe that $P(\Sigma)$ is graded, as it is a lower order ideal of the graded poset $L(BC_n)$. Recall that the $i$'th Whitney number of the second kind of a graded poset is the number of elements of rank $i$.

\begin{proposition}\label{explicitChrom}\normalfont The chromatic function $\chi_{\Sigma}(k)$ of a signed graph $\Sigma$ with vertex set $[n]$ and no positive or free loops is a polynomial given by
\[\chi_\Sigma(k)=\sum_{s\in P(\Sigma)}2^{n-\rho(s)}(k)_{n-\rho(s)}=\sum_{i=0}^nW_{n-i}2^i(k)_i,\]
where $\rho$ is the rank function and $W_i$ is the $i$th Whitney number of the second kind for the poset $P(\Sigma)$, and $(k)_i:=k(k-1)\cdots(k-i+1)$. 
\end{proposition}
\begin{proof}A coloring is proper if and only if it avoids the signed graphic arrangement, thus integer points on flats not contained by hyperplanes of type $h_e^{\sigma(e)}$ are precisely the proper colorings. To count signed $k$-colorings, notice that on a flat of dimension $i$ there are $\binom{k}{i}$ choices for the magnitudes of the coefficients, $2^i$ ways to choose the signs, and $i!$ orderings. Thus on each flat of rank $i$, there are $2^{n-i}(k)_{n-i}$ signed $k$-colorings. Summing by co-dimension gives the second equality.
\end{proof}

\subsection{Group Actions}
The (hyperoctahedral) group of signed permutations may be described by its action on $\R^n$ as being the group generated by reflection about the hyperplanes of $BC_n$. It is sufficient to generate by reflections of type $h_{ij}^+$ and $h_{i}$ \cite{Coxeter}, and in particular any signed permutation may be described by a sequence of transpositions of coordinates, followed by a subset of coordinates to be negated. Geometrically, the transposition $(i,j)$  corresponds to reflection about the hyperplane $h_{ij}^+$, and negating the $i$'th coordinate corresponds to reflection about the hyperplane $h_{i}$.
Thus moving forward we will use a canonical representation of a signed permutation $\beta$ in terms of a $(b,\delta)$ pair, with permutation $b\in S_n$ and switching set $\delta\subseteq [n]$, where the switching set is the subset of $[n]$ whose coordinates are to be negated. The action of $\beta$ on the standard basis vector $e_i$ is \[\beta(e_i)=(-1)^{|\delta\cap\{b(i)\}|}e_{b(i)}.\]
\begin{example}
The point $(4,2,-3)\in\R^3$ has image $(3,2,4)$ under the signed permutation $\beta$ defined by first reflecting across the hyperplane $h_{13}^+$, then reflecting across $h_1$. We represent $\beta$ by the $(b,\delta)$-pair $((1,3),\{1\})$.
\end{example}
There is a natural action of a signed permutation $\beta=(b,\delta)$ on a signed graph $\Sigma$. We define $\beta(\Sigma)$ to be the signed graph with vertex set $[n]$ and edge set resulting from permuting the endpoint set $V(e)$ of each edge $e$ of $\Sigma$ by the permutation $b$. For each edge $e'=b(e)$ of $\beta(\Sigma)$, we set $\sigma(e')=(-1)^{|V(e')\cap\delta|}\sigma(e)$.

\begin{example} Let $\Sigma$ be the signed graph on the left in Figure \ref{fig:signed graph action} and $\beta$ be the signed permutation $\{(1)(23)(4),\{234\}\}$. For the edge $e$ with $V(e)=\{1,2\}$, we have $V(e')=\{1,3\}$ and \[\sigma(e')=(-1)^1\sigma(e)=-1\] since $|V(e')\cap\delta|=|\{1,3\}\cap\{234\}|=1$.

\begin{figure}[h!]
\begin{tikzpicture}[scale=2]

\draw (0,0) --(0,1) -- (1,1) -- (1,0) -- (0,0);
\draw (0,1) -- (1,0);
\draw [fill] (0,0) circle [radius=0.0125];
\draw [fill] (0,1) circle [radius=0.0125];
\draw [fill] (1,0) circle [radius=0.0125];
\draw [fill] (1,1) circle [radius=0.0125];

\node at (0,0) [below left] {$1$};
\node at (0,1) [above left] {$2$};
\node at (1,0) [below right] {$3$};
\node at (1,1) [above right]{$4$};
\node at (1/2,1/2) [above]{$-$};
\node at (0,1/2) [left] {$+$};
\node at (1/2,0) [below] {$-$};
\node at (1,1/2) [right] {$+$};
\node at (1/2,1) [above]{$-$};
\node at (0,1/2) [right] {$e$};
\node at (1/2,-.5){$\Sigma$};

\draw (3-1/2,0) --(3-1/2,1) -- (4-1/2,1) -- (4-1/2,0) -- (3-1/2,0);
\draw (3-1/2,1) -- (4-1/2,0);
\draw [fill] (3-1/2,0) circle [radius=0.0125];
\draw [fill] (3-1/2,1) circle [radius=0.0125];
\draw [fill] (4-1/2,0) circle [radius=0.0125];
\draw [fill] (4-1/2,1) circle [radius=0.0125];

\node at (3-1/2,0) [below left] {$1$};
\node at (3-1/2,1) [above left] {$2$};
\node at (4-1/2,0) [below right] {$3$};
\node at (4-1/2,1) [above right]{$4$};
\node at (3,1/2) [above]{$-$};
\node at (3-1/2,1/2) [left] {$+$};
\node at (3,0) [below] {$-$};
\node at (4-1/2,1/2) [right] {$-$};
\node at (3,1) [above]{$+$};
\node at (3,0) [above] {$e'$};

\node at (3,-0.5) {$\beta(\Sigma)$};

\end{tikzpicture}
\caption{}
  \label{fig:signed graph action}
\end{figure}
\end{example}

We define a $\Sigma$-automorphism to be a signed permutation under which $\Sigma$ is fixed. Observe that ~$\B_\Sigma$ is invariant under the action of a $\Sigma$-automorphism. 

\subsection{Quotient Signed Graphs}Let $\Sigma$ be a signed graph with vertex set $[n]$ and $b\in S_n$ have cycle decomposition $C_1\cdots C_m$, with cycles ordered by minimal element. Given a signed permutation $\beta=(b,\delta)$ and an integer $i$ in $[n]$, let $k$ be the minimal element of $C_s$, the $b$-cycle containing $i$. There is a minimal positive integer $\ell$ such that $b^\ell(k)=i$. 

We define
\[
\beta_{(i)}:=|\delta\cap\{b(k),b^2(k),\dots,b^\ell(k)=i\}|
,\] the number of switching indices of $\delta$ ``between" $k$ and $i$ in $C_s$. In particular, if $k=i$, then $\beta_{(i)}$ is the cardinality of $\delta\cap C_s$.

For a signed permutation (not necessarily a $\Sigma$-automorphism) $\beta=(b,\delta)$, we define the quotient signed graph $\Sigma/ \beta$ to be the signed graph whose vertex set is the collection of $s\in[m]$ such that the intersection $\delta\cap C_s$ has even cardinality. The edge set of $\Sigma/ \beta$ is constructed as follows:  

Let $e$ be an arc of $\Sigma$ with $V(e)=\{i,j\}$ such that $i\in C_s$ and $j\in C_t$. If  $s$ and $t$ are distinct and the cardinalities of $\delta\cap C_s$ and $\delta\cap C_t$ are even, then there is a corresponding arc $e'$ in $\Sigma/\beta$ with $V(e')=\{s,t\}$. If $s$ and $t$ are distinct and the cardinalities of $\delta\cap C_s$ and $\delta\cap C_t$ are even and odd respectively, then there is a corresponding half-edge $e'$ in $\Sigma/\beta$ with $V(e')=\{s\}$. If $s$ and $t$ are equal and the cardinality of $\delta\cap C_s$ is even, then there is a corresponding loop $e'$ in $\Sigma/\beta$ with $V(e')=\{s\}$. If the cardinalities of $\delta\cap C_s$ and $\delta\cap C_t$ are both odd, then there is a corresponding free loop in $\Sigma/\beta$. The last case results in a signed graph with no proper colorings, and so will not arise in the context of this paper.

Edges of $\Sigma/\beta$ arising from loops or half-edges are defined similarly, in that if they have endpoint set $\{i\}$, where $i\in C_s$, then they stay loops or half-edges if the cardinality of $\delta\cap C_s$ is even, and become free loops otherwise.

The sign of each edge of $\Sigma/\beta$ is determined as follows: 

For each edge $e'$ of $\Sigma/\beta$ arising from link $e$ of $\Sigma$ with $V(e)=\{i,j\}$, let the sign of $e'$ be given by \begin{equation}\sigma(e')=\sigma(e)\cdot(-1)^{\beta_{(i)}+\beta_{(j)}}.\end{equation}\label{equationSigns}

For each loop $e'$ of $\Sigma/\beta$ arising from a loop $e$ of $\Sigma$, let the sign of $e'$ be given by \[\sigma(e')=\sigma(e).\]

\begin{example}Given the signed graph $\Sigma$ as on the left in Figure \ref{fig:signed quotient graphs} and a signed permutation \\ ${\beta=\{(1)(23)(4),\{234\}\}}$, we construct the quotient signed graph $\Sigma/\beta$ (right). 

\vspace{.25in}
\begin{figure}[h!]
\begin{tikzpicture}[scale=2]

\draw (0,0) --(0,1) -- (1,1) -- (1,0) -- (0,0);
\draw (0,1) -- (1,0);
\draw [fill] (0,0) circle [radius=0.0125];
\draw [fill] (0,1) circle [radius=0.0125];
\draw [fill] (1,0) circle [radius=0.0125];
\draw [fill] (1,1) circle [radius=0.0125];

\node at (0,0) [below left] {$1$};
\node at (0,1) [above left] {$2$};
\node at (1,0) [below right] {$3$};
\node at (1,1) [above right]{$4$};
\node at (1/2,1/2) [above]{$-$};
\node at (1/2,1/2) [below]{$e_2$};
\node at (0,1/2) [left] {$+$};
\node at (1/2,0) [below] {$e_1$};
\node at (1/2,0) [above] {$-$};
\node at (1,1/2) [right] {$+$};
\node at (1/2,1) [above]{$-$};
\node at (1/2,-.50) {$\Sigma$};

\draw (3,0.5) to [out=45,in=135] (4,0.5) to (4.5,0.75);
\draw (3,0.5) to [out=-45,in=-135] (4,0.5);
\draw (4,0.5) to (4.5,.25);
\draw [fill] (3,0.5) circle [radius=0.0125];
\draw [fill] (4,0.5) circle [radius=0.0125];

\node at (3,0.5) [above] {$1$};
\node at (4,0.5) [above] {$2$};
\node at (3.5,0.675) [above]{$+$};
\node at (3.5,0.3125) [below]{$e_1'$};
\node at (3.5,0.3125) [above]{$-$};

\draw (4,0.5) to [out=-45,in=0] (4,0);
\draw (4,0.5) to [out=180+45,in=180] (4,0);
\node at (4,0.15) {$-$};
\node at (4,0.05) [below]{$e_2'$};
\node at (4,-.50) {$\Sigma/\beta$};

\end{tikzpicture}
\caption{}
  \label{fig:signed quotient graphs}
\end{figure}

The $b$-cycles $C_1=(1)$ and $C_2=(23)$ have even intersection with $\delta$, and thus the vertex set of $\Sigma/\beta$ is $\{1,2\}$. 

We compute that $\beta_{(2)}=2$ since $|\{234\}\cap\{b(2),b^2(2)\}|=|\{234\}\cap\{3,2\}|=2$. 

The endpoint set of $e_1'$ is $\{1,2\}$ since the endpoints of $e_1$ lie in $C_1$ and $C_2$, respectively. 

We compute $\sigma(e_1')$ by 
\[\sigma(e_1')=\sigma(e_1)\cdot(-1)^{\beta_{(1)}+\beta_{(2)}}=(-1)\cdot(-1)^{0+2}=-1.\]
The endpoint set of edge $e_2$ is $\{2,3\}$, a subset of $C_2$. Thus the loop $e_2'$ has \[\sigma(e_2')=\sigma(e_2)(-1)^{\beta_{(2)}+\beta_{(2)}}=-1.\]
\end{example}

\section{Unlabeled Signed Chromatic Polynomials}
To a signed permutation $\beta=(b,\delta)$ we associate the flat $\widehat{\beta}\in L(BC_n)$ by
\[
\widehat{\beta}:=\bigcap_{i\in[n]} \{{\bf{x}}\in\R^n:\;x_i=(-1)^{|\delta\cap\{i\}|}x_{b^{-1}(i)}\}.
\] The set $\widehat{\beta}$ is defined so that it is precisely the collection of fixed points of the linear map $\beta$. 
\begin{lemma}\label{fixedPoint}
\normalfont A signed coloring $\gamma$ is fixed by an automorphism $\beta$ if and only if it is an element of the geometric set $\widehat{\beta}$.
\end{lemma}
\begin{proof} A point ${\bf{x}}\in\R^n$ is fixed by the action of $\beta$ if and only if for all $i\in[n]$, the dot product $e_i\cdot\beta({\bf{x}})=x_i$. Observe that 
\begin{align*}
x_i=e_i\cdot\beta({\bf{x}})=e_i\cdot\left(\sum_{j=1}^nx_j\beta(e_j)\right)&=e_i\cdot\left(\sum_{j=1}^nx_j(-1)^{|\delta\cap \{b(j)\}|}e_{b(j)}\right)\\
&=(-1)^{|\delta\cap \{i\}|}x_{b^{-1}(i)},\end{align*}so that ${\bf{x}}$ is in $\widehat{\beta}$.
\end{proof}
Note that for ${\bf{x}}\in\widehat{\beta}$ and $i\in[n]$, we get the following chain of equalities:
\begin{align*}
x_i&=(-1)^{|\delta\cap\{i\}|}x_{b^{-1}(i)}\\
x_{b^{-1}(i)}&=(-1)^{|\delta\cap\{b^{-1}(i)\}|}x_{b^{-2}(i)}\\
&\vdots\\
x_{b(i)}&=(-1)^{|\delta\cap\{b(i)\}|}x_{i},
\end{align*} so that for $i\in C_s$, with $k=\minimum\{C_s\}$,
\begin{align}x_i=(-1)^{\beta_{(i)}}x_k.\end{align}\label{equationMin} Thus ${\bf{x}}\in\widehat{\beta}$ is determined by the coordinate indexed by the minimal element of each cycle of $b$. Observe  that for $k$ the minimal element of $C_s$, we get the equality $x_k=(-1)^{|\delta\cap C_s|}x_k$, so that if $ |\delta\cap C_s|$ is odd, then $x_k=0$. Indeed, in this case $x_i=0$ for all $i$ in $C_s$.
\begin{lemma}\label{quotient}\normalfont The number of proper signed $k$-colorings of $\Sigma$ contained in flat $\widehat{\beta}$ is given by $\chi_{\Sigma/\beta}(k)$.
\end{lemma}
\begin{proof} We proceed in two steps. First we present a set map from proper $\Sigma$ colorings contained in $\widehat{\beta}$ to proper $\Sigma/\beta$ colorings. Next we show that this map is bijective.

\begin{enumerate}\item Given a proper $\Sigma$-coloring $\gamma\in\widehat{\beta}$, we construct a  $\Sigma/\beta$-coloring $\gamma'$ by letting \[\gamma'(s)=\gamma(\minimum\{C_s\})\] for each vertex $s$ of $\Sigma/\beta$. We now prove that $\gamma'$ is a proper $\Sigma/\beta$ coloring using equations (1) and (2).

We check cases, noting that edges of $\Sigma$ and $\Sigma/\beta$ are in bijection. For edge $e'$ of $\Sigma/\beta$ arising from edge $e$ of $\Sigma$:
\begin{enumerate}[(i)]

\item If $e'$ is a link, then apply equations (1)  and (2) to show that $\gamma'$ is proper with respect to $e'$.

\item If $e'$ is a half-edge with endpoint $s$, then $\gamma'$ is proper with respect to $e'$ exactly when $\gamma'(s)\neq0$. If $e$ is a half-edge with endpoint $i\in C_s$, then properness of $\gamma$ implies that $\gamma(i)\neq0$, and thus by equation (2) we see that $\gamma'(s)=\pm\gamma(i)\neq0$. If instead $e$ has distinct endpoints $i$ and $j$, where $j$ is in a $b$-cycle whose intersection with $\delta$ has odd cardinality, then since $\gamma$ is in $\widehat{\beta}$, we have by equation (2) that $\gamma(j)=0$. Properness of $\gamma$ implies that $\gamma(i)\neq0$, so that $\gamma'(s)=\pm\gamma(i)\neq0$ and $\gamma'$ is proper with respect to $e'$.

\item If $e'$ is a loop at vertex $s$ and $e$ is a loop at vertex $i\in C_s$, then since $\Sigma$ is properly colored by $\gamma$, which is in $\widehat{\beta}$,  we have that $\gamma'(s)=\pm\gamma(i)\neq0$. It follows that $\gamma'$ is proper with respect to $e'$. If instead $e$ has distinct endpoints $i$ and $j$, both in $C_s$, then we have that $\gamma(i)=\pm\gamma(j)$, and by properness for edge $e$ we see that neither $\gamma(i)$ nor $\gamma(j)$ is equal to zero. Thus $\gamma'(s)=\pm\gamma(i)\neq0$ and $\gamma'$ is proper with respect to $e'$. These are the only two cases for $e$ such that $e'$ will be a loop.
\end{enumerate}
Thus $\gamma'$ is a proper $\Sigma/\beta$ coloring.

\item Given $\gamma'$ a proper $\Sigma/\beta$-coloring, we find its preimage $\gamma\in\widehat{\beta}$ as follows. For each vertex $s$ of $\Sigma/\beta$, set $\gamma(\minimum\{C_s\})=\gamma'(s)$, and define the value of $\gamma$ for each $i$ in $C_s$ by recursively using the equation $\gamma(b(i))=(-1)^{|\delta\cap\{b(i)\}|}\gamma(i)$. For those vertices $j$ of $\Sigma$ with $j\in C_t$ and $t$ not a vertex of $\Sigma/\beta$ (i.e., $\delta\cap C_t$ has odd cardinality), set $\gamma(j)=0$. So defined, the coloring $\gamma$ is an element of $\widehat{\beta}$.

It remains to show that $\gamma$ is a proper coloring of $\Sigma$.

 Again we check cases, noting that there is a bijection between edges of $\Sigma$ and those of $\Sigma/\beta$.
\begin{enumerate}[(i)]
\item If an edge $e'$ of $\Sigma/\beta$ has distinct endpoints $s$ and $t$, then apply equations (1) and (2) to show that $\gamma$ is proper with respect to $e$. 

\item If $e'$ is a loop at vertex $s$ and $e$ is a loop at vertex $i\in C_s$, then since $\gamma'$ is proper, $\gamma(i)=\pm\gamma'(s)\neq0$ and the coloring $\gamma$ is proper with respect to $e$. If instead $e$ has distinct endpoints $i$ and $j$, then by equation (2), $\gamma(i)=(-1)^{\beta_{(i)}+\beta_{(j)}}\gamma(j)$. By equation (1), we have $\sigma(e)=\sigma(e')(-1)^{\beta_{(i)}+\beta_{(j)}}$. Since $\gamma'$ is proper, $\sigma(e')=-1$ and $\gamma(i)$ and $\gamma(j)$ are nonzero, so that $\gamma$ is proper with respect to $e$.
\item If $e'$ is a half-edge at vertex $s$ and $e$ a half-edge at vertex $i$ (with $i\in C_s$), then $\gamma(i)=\pm\gamma'(s)\neq0$ and $\gamma$ is proper with respect to $e$. If instead $e$ has distinct endpoints $i$ and $j$, where $j$ is in a $b$-cycle having odd intersection with $\delta$, then $\gamma(i)=\pm\gamma'(s)\neq0$ and  $\gamma(j)=0$, so that the coloring $\gamma$ is proper with respect to $e$.\end{enumerate}\end{enumerate}Thus $\gamma\in\widehat{\beta}$ is a proper $\Sigma$-coloring and we have established that proper colorings of $\Sigma$ contained in $\widehat{\beta}$ are in bijection with proper colorings of $\Sigma/\beta$. \end{proof}

Hanlon discusses counting proper $k$-colorings of {\it{unlabeled}} graphs in the paper \cite{HanlonUnlabeled}. Due to our geometric viewpoint, we find it more convenient to discuss proper colorings of a labeled graph, where two colorings $\gamma$ and $\gamma'$ are considered equivalent if there is an automorphism $\beta$ of the graph such that $\beta(\gamma)=\gamma'$ (as points in $\Z^n$).
The unlabeled chromatic function $\widehat{\chi}_{\Sigma}$ of a signed graph $\Sigma$ with vertex set $[n]$ is a function whose input is a natural number $k$ and whose output is the number of proper colorings of $\Sigma$ {\it{up to automorphism}} taking values in $[-k,k]^n\cap\Z^n$.
\begin{theorem} \label{mainunlabeledtheorem}\normalfont For a signed graph $\Sigma$ with automorphism group $B$, the unlabeled chromatic function $\widehat{\chi}_\Sigma(k)$ is a polynomial given by
\[\widehat{\chi}_{\Sigma}(k)=\frac{1}{|B|}\sum_{\beta\in B}\chi_{\Sigma/\beta}(k).\]
\end{theorem}
\begin{proof}
By Burnside's lemma, the number of orbits in the set of proper $k$-colorings of $\Sigma$ under the action of $B$ is given by the sum \[
\frac{1}{|B|}\sum_{\beta\in B}|\fix(\beta)|,\] where $\fix(\beta)$ is the set of proper $k$-colorings of $\Sigma$ which are invariant under the action of $\beta$.
Applying Lemmas 3.1 and 3.2, we see that the summands are given by $\chi_{\Sigma/\beta}(k)$, each of which is a polynomial due to a result of Zaslavsky (and demonstrated in Proposition \ref{explicitChrom} above). 
\end{proof}

It is possible to compute $\widehat{\chi}_\Sigma(k)$ using Theorem 3.3 and Proposition \ref{explicitChrom}, but this requires a rather tedious process of computing $\Sigma/\beta$ for each $\beta\in B$, and then forming the poset $P(\Sigma/\beta)$ in order to compute $\chi_{\Sigma/\beta}(k)$. Instead, the following simplified calculation can be used.

\begin{theorem}\label{explicitUnlab}\normalfont For a signed graph $\Sigma$ without free or positive loops and with automorphism group $B$, the unlabeled chromatic polynomial $\widehat{\chi}_\Sigma(k)$ is given by
\[\widehat{\chi}_\Sigma(k)=\frac{1}{|B|}\sum_{\beta\in B}\;\sum_{\substack{s\in P(\Sigma)\\s\subseteq\widehat{\beta}}}2^{n-\rho(s)}(k)_{n-\rho(s)},\]where $\rho$ is the rank function of the poset $P(\Sigma)$. 
\end{theorem}
\begin{proof}It is sufficient to recall Lemma \ref{quotient} and observe that the inner sum is over flats contained by $\widehat{\beta}$ and not contained by the signed graphic arrangement. 
\end{proof}
\section{Unlabeled Signed Acyclic Orientations} There is a well known reciprocity, established by Stanley, between the chromatic polynomial of an ordinary graph and the number of acyclic orientations of the graph. A similar connection was presented by Hanlon in the case of unlabeled graphs, and we extend this to the case of signed graphs.
 
Similar to the orienting of edges of ordinary graphs, {\it{incidences}} of signed graphs may be oriented, i.e., for a given edge $e$ with endpoint $i$, we orient the incidence by $\tau(e,i)=\pm1$. An orientation of a signed graph is an orientation of each incidence (free loops have no incidences, half-edges have one, and all other edges have two) subject to the following constraint on each edge $e$ with endpoints $i$ and $j$:
\[\tau(e,i)=-\sigma(e)\tau(e,j).\] 

An orientation of a signed graph $\Sigma$ is called acyclic if every simple cycle has a source or sink. Zaslavsky showed in the paper \cite{ZaslavskyOrientationsSigned} that acyclic orientations of a signed graph $\Sigma$ are in bijection with maximal connected components (regions) of $\R^n\backslash\B_{\Sigma}$. This implies that the number of acyclic orientations of a signed graph $\Sigma$ is given by $(-1)^{V(\Sigma)}\chi_\Sigma(-1)$. Let $\Delta(\Sigma)$ denote the set of acyclic orientations of $\Sigma$, and  for ~$\tau\in\Delta(\Sigma)$, let $r(\tau)$ be the associated region.

The action of a $\Sigma$-automorphism $\beta$ on $\R^n$ induces an action on the regions of $\B_{\Sigma}$. This defines an action on the set of orientations $\tau$.
\begin{lemma}\label{fixedOrientations}\normalfont An acyclic orientation $\tau\in\Delta(\Sigma)$ is fixed by a $\Sigma$-automorphism $\beta$ if and only if $r(\tau)\cap\widehat{\beta}\neq\emptyset$.
\end{lemma}
\begin{proof}
Let $\beta$ fix $r(\tau)$. Then for $y\in r(\tau)$, the orbit of $y$ under powers of $\beta$ is a subset of $r(\tau)$, and thus the average $\overline{y}$ of this orbit is in the convex set $r(\tau)$ and is fixed by $\beta$. By Lemma \ref{fixedPoint}, a point in $\R^n$ is fixed by $\beta$ if and only if it is contained by $\widehat{\beta}$, thus $r(\tau)\cap \widehat{\beta}$ is non empty since it contains $\overline{y}$.

Conversely, let $y\in r(\tau)\cap \widehat{\beta}$. As a linear map, $\beta$ is continuous and so preserves path connectedness of $r(\tau)$. Thus for all $y'$ in $\beta\cdot r(\tau)$ there exists a path from $y$ to $y'$ in the complement of $\B_{\Sigma}$. Since $r(\tau)$ is defined to be a maximal connected subset of the complement of $\B_{\Sigma}$, $y'$ is in $r(\tau)$, showing that $\beta\cdot r(\tau)\subseteq r(\tau)$ and hence that $r(\tau)$ is fixed by $\beta$.\end{proof}

We define an unlabeled acyclic orientation of a signed graph to be an acyclic orientation {\it{up to automorphism}}, i.e., two orientations $\tau$ and $\tau'$ are considered equivalent if there exists $\Sigma$-automorphism $\beta$ such that $\beta\cdot r(\tau)=r(\tau')$.
\begin{theorem}\label{unlabeledacycliccounting}\normalfont The set $\widehat{\Delta}$ of unlabeled acyclic orientations of a signed graph $\Sigma$ with automorphism group $B$ has cardinality given by
\[|\widehat{\Delta}|=\frac{1}{|B|}\sum_{\beta\in B}(-1)^{V(\Sigma/\beta)}\chi_{\Sigma/\beta}(-1).\] 
\end{theorem}
\begin{proof}The proper colorings of $\Sigma$ contained in $\widehat{\beta}$ are precisely the lattice points in $\widehat{\beta}\backslash\B_{\Sigma}$. It follows from Theorem 2.2 of \cite{finiteField} and Lemma \ref{quotient} that $\chi_{\Sigma/\beta}(k)$ is the characteristic polynomial of the induced arrangement on $\widehat{\beta}$, and that the number of regions $r(\tau)$ of $\B_\Sigma$ meeting $\widehat{\beta}$ is given by $(-1)^{V(\Sigma/\beta)}\chi_{\Sigma/\beta}(-1)$.

Applying Burnside's Lemma as in the proof of Theorem \ref{mainunlabeledtheorem}, the result follows from Lemma \ref{fixedOrientations}.\end{proof}

\section*{Acknowledgements}The author thanks Matthias Beck for suggesting and supervising this work, as well as the referee for their helpful comments and suggestions. This project was partially supported by the NSF GK-12 program (grant DGE-0841164).

\providecommand{\bysame}{\leavevmode\hbox to3em{\hrulefill}\thinspace}
\providecommand{\MR}{\relax\ifhmode\unskip\space\fi MR }
\providecommand{\MRhref}[2]{%
  \href{http://www.ams.org/mathscinet-getitem?mr=#1}{#2}
}
\providecommand{\href}[2]{#2}

%

\end{document}